\documentclass{amsart}

\usepackage{amsmath, amsthm, amssymb, xcolor, graphicx, amsfonts, url} 
\usepackage[top=1in, bottom=1in, left=1.2in, right=1.2in]{geometry}

\usepackage[style=alphabetic]{biblatex}
\newcommand{\norm}[1]{\left\|#1\right\|}
\newcommand{\inner}[2]{\left\langle #1, #2 \right\rangle}

\newcommand{\e}[0]{\varepsilon}

\newcommand{\K}[0]{\mathcal{K}}
\renewcommand{\L}[0]{\mathcal{L}}
\renewcommand{\H}[0]{\mathcal{H}}

\newtheorem{theorem}{Theorem}[section]
\newtheorem{lemma}[theorem]{Lemma}
\newtheorem{proposition}[theorem]{Proposition}
\newtheorem{corollary}[theorem]{Corollary}
\theoremstyle{definition}
\newtheorem{definition}[theorem]{Definition}
\theoremstyle{remark}
\newtheorem{remark}[theorem]{Remark}
\newtheorem{example}[theorem]{Example}

\title{Flattening and asymptotic orthogonalization of Completely Positive Maps}
\author{Yoonje Jeong}
\address{Department of Mathematics, University of California San Diego, 9500 Gilman Drive, La Jolla, CA 92093, USA}
\email{yojeong@ucsd.edu}
\date{\today}

\begin{document}

\begin{abstract}
    Let $M$ be a $\mathrm{II}_1$ factor, $N$ a tracial von Neumann algebra, and $\Phi: M \rightarrow N$ a subtracial completely positive map. For an irreducible $\mathrm{II}_1$ subfactor $P \subseteq M$, we characterize when $\Phi$ exhibits a \textit{flattening} property under conjugation by unitaries in $P$. To be specific, we show that the failure of a Pimsner-Popa type inequality for $E_P \circ \Phi^* \circ \Phi$ is the precise obstruction, equivalently characterized by left weak mixing of a naturally associated $P$-$N$ bimodule. As an application, we obtain an asymptotic orthogonalization result generalizing a result of Popa \cite{Po19}.
\end{abstract}

\maketitle

\section{Introduction}

Let $A$ and $B$ be von Neumann subalgebras of a $\mathrm{II}_1$ factor $M$. The \textit{orthogonalization problem} for the pair $A, B$ asks whether there exists a unitary $u \in M$ such that $uAu^*$ and $B$ are orthogonal in the sense that
\begin{equation*}
    \tau(uau^*b) = \tau(a)\tau(b)
\end{equation*}
for all $a \in A$ and $b \in B$. Equivalently, this amounts to finding a unitary $u \in M$ satisfying
\begin{equation*}
    u(A \ominus \mathbb{C})u^* \perp B \subseteq L^2(M).
\end{equation*}
For instance, any unitary $u \in M$ that is free from $B$ automatically satisfies this orthogonality relation. The orthogonalization problem provides a natural source of invariants for the inclusion $A \subseteq M$, and has been studied, among others, in \cite{Po83} and \cite{Po95}. More recently, Popa proved in \cite[Corollary 1.2]{Po19} the following. Let $\omega$ be a free ultrafilter on $\mathbb{N}$, and let $Q \subseteq M$ be a von Neumann subalgebra satisfying $M \not\prec_M Q$ in the sense of \cite{Po06a} and \cite{Po06b} (see Definition \ref{def-xcorner}). If $B \subseteq M$ is a separable von Neumann subalgebra, then there exists a unitary $u \in M^\omega$ such that $uBu^*$ and $Q^\omega$ are orthogonal.

A key ingredient in the proof is \cite[Lemma~2.3]{Po19}, which furnishes a sufficient condition for the \textit{flattening} of elements of $M$ onto their trace. Precisely, it asserts that whenever $M \nprec_M Q$, there exists a net $(u_i)_{i \in I} \subseteq \mathcal{U}(M)$ such that
\begin{equation*}
    \lim_{i \in I} \norm{E_Q(u_i x u_i^*) - \tau(x)}_2 = 0
\end{equation*}
for all $x \in M$. He proved this result using the so-called `patching' technique similar to his previous paper \cite{Po95}. In this paper, using different approaches based on Popa's Local Quantization Principle and an integration technique (see subsection \ref{new}), we establish a necessary and sufficient condition for this phenomenon to occur in a more general setting of a subtracial completely positive map $\Phi$ from a $\mathrm{II}_1$ factor $M$ to an arbitrary tracial von Neumann algebra $N$. This recovers the above sufficient condition in \cite[Lemma 2.3]{Po19} as a special case. For instance, we may take $\Phi$ to be a \textit{quantum channel}, i.e., a trace-preserving completely positive map. To be precise, we will prove the following result:

\begin{theorem} \label{main}
    Let $M$ be a $\mathrm{II}_1$ factor and $N$ a tracial von Neumann algebra. Suppose that $\Phi: M \rightarrow N$ is a subtracial completely positive map and that $P \subseteq M$ is an irreducible (i.e., $P' \cap M = \mathbb{C}$) $\mathrm{II}_1$ subfactor of $M$. Then the following are equivalent:
    \begin{itemize}
        \item[(a)] There exists a net $(u_i)_{i \in I} \subseteq \mathcal{U}(P)$ such that
        \begin{equation*}
            \lim_{i \in I} \norm{\Phi(u_i x u_i ^*) - \tau(x) \Phi(1)}_2 = 0
        \end{equation*}
        for all $x \in M$.
    
        \item[(b)] There is no $a \in P$ with $\tau(a) \neq 0$ such that
        \begin{equation*}
            (E_P \circ \Phi^* \circ \Phi)(x) \geq a^* x a
        \end{equation*}
        for all $x \in P_+$.
        
        \item[(c)] The $P\text{-}N$ bimodule,
        \begin{equation*}
            \overline{P \xi_\Phi N} = \overline{\operatorname{span}} \{x\xi_\Phi y : x \in P , y \in N\} \simeq \mathcal{H}(\Phi|_P)
        \end{equation*}
        is left weakly mixing.
    \end{itemize}
    Moreover, if $M$ is separable, we may replace the net in condition (a) by a sequence.
\end{theorem}

Here we denote $\mathcal{H}(\Psi)$ by the bimodule associated with a completely positive map $\Psi$. We refer to subsection \ref{protein} for the precise definition.

\begin{remark}
    Condition (a) in Theorem \ref{main} has the following obvious reformulation as a flattening condition.
    \begin{itemize}
        \item[(a')] For every finite subset $F$ of $M$ such that $\tau(x) = 0$ for all $x \in F$, we have
        \begin{equation*}
            \inf_{u\in \mathcal{U}(P)} \max_{x \in F} \norm{\Phi(u x u ^*)}_2 = 0.
        \end{equation*}
    \end{itemize}
    We will often use this condition instead of (a) without explicitly mentioning it.
\end{remark}

Condition (b) may be viewed as a generalized bimodule-theoretic negation of the Pimsner-Popa inequality. Recall that for a subfactor $Q \subseteq M$, the Pimsner-Popa inequality \cite{PP86} asserts that
\begin{equation*}
    E_Q(x) \geq [M:Q]^{-1} x
\end{equation*}
for all $x \in M_+$. In this light, Theorem~\ref{main} identifies the failure of a Pimsner-Popa type bound as the precise algebraic obstruction to flattening; see the discussion following Corollary~\ref{suff}. The notion of a left weakly mixing bimodule in condition (c), which extends the classical notion of a weakly mixing group action, was introduced by Peterson and Sinclair \cite{PS12}. In the special case $P = M$, Theorem~\ref{main} specializes to the following.
\begin{corollary} \label{cormain}
    Let $M$ be a $\mathrm{II}_1$ factor, $N$ be a tracial von Neumann algebra, and $\Phi: M \rightarrow N$ be a subtracial completely positive map. Then the following are equivalent:
    \begin{itemize}
        \item[(a)] There exists a net $(u_i)_{i \in I} \subseteq \mathcal{U}(M)$ such that
        \begin{equation*}
            \lim_{i \in I} \norm{\Phi(u_i x u_i ^*) - \tau(x) \Phi(1)}_2 = 0
        \end{equation*}
        for all $x \in M$. (If $M$ is separable, we may replace the net by a sequence.)

        \item[(b)] $\mathcal{H}(\Phi)$ is left weakly mixing as an $M$-$N$ bimodule.
    \end{itemize}
\end{corollary}

\subsection{Organization of the paper}

Section~\ref{sec-pre} recalls the necessary preliminaries on bimodules of tracial von Neumann algebras, subtracial completely positive maps, and their interrelation, together with basic facts about central vectors and weakly mixing bimodules. In Section~\ref{sec-encoding}, we show that for a bimodule with a cyclic vector, weak mixing is entirely determined by the central part of a single vector (Lemma~\ref{encoding} and Proposition~\ref{central}). Section~\ref{sec-proof} is devoted to the proof of Theorem~\ref{main}; the implication (c)$\Rightarrow$(a) rests on a local quantization argument for bimodules, recently introduced in \cite{QIT26}. In Section~5, we apply Theorem~\ref{main} to obtain an asymptotic orthogonalization result (Corollary~\ref{asym-otho}) for subtracial completely positive maps, extending \cite[Corollary~1.2]{Po19}. Section~\ref{sec-moreapp} collects further applications, including an application to the closure of unitary orbits of subfactors with respect to the Maréchal topology.

\subsection{Acknowledgements and AI tool disclosure.}

The author thanks his advisor Adrian Ioana for suggesting the problem studied in this paper and for his guidance throughout its development. Claude and ChatGPT were used for English language editing, proofreading, and grammatical corrections. In addition, the author used ChatGPT to identify a relevant paper \cite{Choda96}, which motivated the use of outer actions in constructing the finite-index example in Section~\ref{sec-moreapp} regarding the Maréchal topology.

\section{Preliminaries} \label{sec-pre}

\subsection{Tracial von Neumann algebras}

We briefly recall some basic terminology concerning tracial von Neumann algebras; we refer the reader to \cite[Chapter~8]{AP17} for further details. A \textit{tracial von Neumann algebra} is a pair $(M,\tau)$ consisting of a von Neumann algebra $M$ and a normal faithful tracial state $\tau:M\rightarrow\mathbb{C}$. For $x\in M$ and $1\leq p<\infty$, we write
\begin{equation*}
    \norm{x}_p=\tau(|x|^p)^{1/p},
\end{equation*}
where $|x|=(x^*x)^{1/2}$. We denote by $L^2(M)$ the Hilbert space completion of $M$ with respect to $\norm{\cdot}_2$, and identify $M$ with its canonical dense subspace in $L^2(M)$. The inner product on $L^2(M)$ is given by
\begin{equation*}
    \inner{x}{y}=\tau(x^*y), \qquad x,y\in M.
\end{equation*}
We write $\mathcal{U}(M)$ for the unitary group of $M$ and $M_1$ for the operator norm unit ball of $M$.

If $N\subseteq M$ is a von Neumann subalgebra, we always assume that the inclusion is unital and that $N$ is endowed with the restriction of $\tau$. We denote by $E_N:M\rightarrow N$ the unique trace-preserving conditional expectation. It is the restriction of the orthogonal projection $e_N:L^2(M)\rightarrow L^2(N)$ to $M$. A tracial von Neumann algebra $(M,\tau)$ is called a $\mathrm{II}_1$ factor if its center is trivial, i.e., $\mathcal{Z}(M)=\mathbb{C}$.

\subsection{Bimodules, bounded vectors, and tensor products}

Throughout this subsection, $M$, $N$, and $Q$ denote tracial von Neumann algebras; we refer the reader to \cite[Section~8.4]{AP17} for details.

\begin{definition}
    Let $\H$ be a right $N$-module. A vector $\xi \in \H$ is \textit{left $N$-bounded} if the map
    \begin{equation*}
        N \rightarrow \H : x \mapsto \xi x
    \end{equation*}
    extends to a bounded operator $L_\xi : L^2(N) \rightarrow \H$. We denote by $\H^0$ the set of left $N$-bounded vectors in $\H$. The set $\H^0$ is dense in $\H$; see \cite[Proposition~8.4.4]{AP17}.
\end{definition}

If $\eta$ and $\xi$ are left $N$-bounded vectors of a right $N$-module $\H$, then
\begin{equation*}
    L_\eta^* L_\xi : L^2(N) \rightarrow L^2(N)
\end{equation*}
is a bounded operator commuting with the right action of $N$ on $L^2(N)$. Consequently, we have a well-defined sesquilinear map $\inner{\cdot}{\cdot}_N : \H^0 \times \H^0 \rightarrow N$ given by
\begin{equation*}
    \inner{\eta}{\xi}_N = L_\eta^* L_\xi \in N.
\end{equation*}
In this case, we have
\begin{equation*}
    \tau(\inner{\eta}{\xi}_N) = \inner{\widehat{1}}{L_\eta^* L_\xi \widehat{1}} = \inner{\eta}{\xi}
\end{equation*}
for all $\eta, \xi \in \H^0$. We now turn to two constructions on bimodules that will be used throughout the paper: the contragredient bimodule and the tensor product of bimodules.

\begin{definition}
    Given an $M$-$N$ bimodule $\mathcal{H}$, the \textit{contragredient bimodule} is the conjugate Hilbert space $\overline{\mathcal{H}}$ equipped with the actions
    \begin{equation*}
        y \cdot \overline{\xi} \cdot x = \overline{x^{*} \xi y^{*}}
    \end{equation*}
    for $x \in M$ and $y \in N$. In particular, $\overline{\H}$ is an $N$-$M$ bimodule.
\end{definition}

We next define the \textit{(Connes) tensor product} of bimodules. Let $\H$ and $\K$ be $M$-$N$ and $N$-$Q$ bimodules, respectively. We define a sesquilinear form on $\H^0 \odot \K$ by
\begin{equation*}
    \inner{\sum_i \xi_i \otimes \eta_i}{\sum_j \xi_j' \otimes \eta_j'} 
    = \sum_{i,j} \inner{\eta_i}{\inner{\xi_i}{\xi_j'}_N \eta_j'}_\K.
\end{equation*}
Throughout this paper, every inner product is taken to be conjugate-linear in its first argument. We denote by $\mathcal H \otimes_N \K$ the completion of the quotient of $\H^0 \odot \K$ by the null space. It is an $M$-$Q$ bimodule whose actions are given by
\begin{equation*}
    x(\xi \otimes \eta)y = (x\xi) \otimes (\eta y)
\end{equation*}
for $x \in M$ and $y \in Q$. One can easily verify that
\begin{equation*}
    (\xi x) \otimes \eta = \xi \otimes (x\eta)
\end{equation*}
for all $x \in N$.

\subsection{Completely positive maps and bimodules} \label{protein}

Having recalled the basic theory of bimodules, we now turn to their correspondence to completely positive maps. This construction generalizes the familiar $M$-$N$ bimodule $L^2(M)$ associated to a conditional expectation $E_N : M \rightarrow N$. We first explain how bimodules can be constructed from completely positive maps.

\begin{definition}
    A completely positive map $\Phi : M\rightarrow N$ between tracial von Neumann algebras is said to be \textit{subtracial} if $\tau_N \circ \Phi \leq \tau_M$.
\end{definition}

We write $\tau$ in place of $\tau_M$ and $\tau_N$ when the algebras are clear from context. Any subtracial completely positive map $\Phi : M \rightarrow N$ is automatically normal and admits an adjoint, which is a normal completely positive map $\Phi^* : N \rightarrow M$ such that
\begin{equation*}
    \tau(\Phi(x)y) = \tau(x\Phi^*(y))
\end{equation*}
for all $x \in M$ and $y \in N$. For any normal completely positive map $\Phi : M \rightarrow N$ between tracial von Neumann algebras, we have a sesquilinear form $\inner{\cdot}{\cdot}$ on the algebraic tensor product $M \odot N$ given by
\begin{equation*}
    \inner{x_1 \otimes y_1}{x_2 \otimes y_2} = \tau_N(y_1^* \Phi(x_1^* x_2) y_2)
\end{equation*}
for $x_1, x_2 \in M$ and $y_1, y_2 \in N$. We denote by $\mathcal H(\Phi)$ the completion of the quotient of $M \odot N$ by the null space. This Hilbert space carries an $M$-$N$ bimodule structure whose actions are given by
\begin{equation*}
    a (x\otimes y) b = (ax) \otimes (yb)
\end{equation*}
for $a,x \in M$ and $b,y \in N$. Set $\xi_\Phi = 1 \otimes 1 \in \mathcal{H}(\Phi)$, so that
\begin{equation*}
    \tau(\Phi(x)y) = \inner{1 \otimes 1}{x \otimes y} = \inner{\xi_\Phi}{x \xi_\Phi y}
\end{equation*}
for all $x \in M$ and $y \in N$. Note that the vector $\xi_\Phi$ is \textit{cyclic} in the sense that
\begin{equation*}
    \overline{M \xi_\Phi N} = \overline{\operatorname{span}} \{x \xi_\Phi y : x \in M , y \in N\} = \mathcal{H}(\Phi).
\end{equation*}

\begin{remark} \label{embed}
    If $\Phi$ is a subtracial completely positive map, $\Phi^*:N\rightarrow M$ is a well-defined normal completely positive map. Therefore, we can define an $M$-bimodule $\H(\Phi^* \circ \Phi)$, which can be embedded as a subbimodule of the $M$-bimodule
    \begin{equation*}
        \H(\Phi) \otimes _N \H(\Phi^*) \simeq \H(\Phi) \otimes_N \overline{\H(\Phi)}
    \end{equation*}
    with the identification $x \otimes y \mapsto (x\otimes 1) \otimes (1 \otimes y)$. Under this identification, we have $\xi_{\Phi^* \circ \Phi}=\xi_\Phi \otimes \overline{\xi_\Phi}$.
\end{remark}

We next explain the converse direction: bounded vectors of bimodules give rise to completely positive maps. Given a left $N$-bounded vector $\zeta$ of an $M$-$N$ bimodule $\mathcal{H}$, we obtain a corresponding completely positive map $\Phi : M \rightarrow N$ such that
\begin{equation*}
    \tau(\Phi(x) y ) = \inner{\zeta}{x \zeta y}
\end{equation*}
for all $x \in M$ and $y \in N$. Indeed, if we denote by $\pi : M \rightarrow B(\H)$ the unital normal $*$-homomorphism giving the action of $M$ on $\H$, the map
\begin{equation*}
    \Phi : M \rightarrow N : x \mapsto L_\zeta^*\pi(x) L_\zeta
\end{equation*}
satisfies the desired properties. We now refine the above correspondence to the level of subbimodules.

\begin{lemma} \label{sub}
    Let $\Phi : M \rightarrow N$ be a normal completely positive map between tracial von Neumann algebras and let $P \subseteq M$, $Q \subseteq N$ be von Neumann subalgebras. For a $P$-$Q$ subbimodule $\K$ of $\H(\Phi)$, set $\zeta = P_\K(\xi_\Phi)$. Then there exists a completely positive map $\Phi_\K : P \rightarrow Q$ such that
    \begin{equation*}
        \tau(\Phi_\K(x)y) = \inner{\zeta}{x \zeta y}
    \end{equation*}
    for all $x \in P$ and $y \in Q$. Moreover, $E_Q \circ \Phi|_P - \Phi_\K$ is completely positive.
\end{lemma}

\begin{proof}
    The orthogonal projection $P_\K$ from $\H(\Phi)$ onto $\K$ is $P$-$Q$ bilinear. Consequently, we have
    \begin{equation*}
        \norm{\zeta y} = \norm{P_\K(\xi_\Phi) y} = \norm{P_\K(\xi_\Phi y)} \leq \norm{\xi_\Phi y} \leq \norm{\Phi(1)}^{\frac{1}{2}} \norm{y}_2
    \end{equation*}
    for all $y \in Q$. It follows that $\zeta \in \mathcal{K}$ is left $Q$-bounded. In the same way, we see that $\zeta' = P_{\K^\perp}(\xi_\Phi)$ is left $Q$-bounded, where $\K^\perp$ is the orthogonal complement of $\K$ in $\mathcal{H}(\Phi)$. Therefore, we obtain completely positive maps $\Phi_\K, \Phi_{\K^\perp} : P \rightarrow Q$ such that
    \begin{equation*}
        \tau(\Phi_\K(x)y) = \inner{\zeta}{x \zeta y}, \quad \tau(\Phi_{\K^\perp}(x)y) = \inner{\zeta'}{x \zeta' y}
    \end{equation*}
    for all $x \in P$ and $y \in Q$. From $\zeta + \zeta' = \xi_\Phi$, we deduce
    \begin{equation*}
        E_Q \circ \Phi|_P = \Phi_\K + \Phi_{\K^\perp},
    \end{equation*}
    which completes the proof.
\end{proof}

\subsection{Central vectors and weakly mixing bimodules}

We proceed to the notion of central vectors, which will play a key role in detecting weak mixing (see Section \ref{sec-encoding}). Let $M$ be a tracial von Neumann algebra and $\mathcal{L}$ be an $M$-bimodule. A vector $\xi \in \mathcal{L}$ is \textit{$M$-central} if $x \xi = \xi x$ for all $x \in M$. We denote by $\mathcal{L}^M$ the set of central vectors in $\mathcal{L}$; this is a closed subspace of $\mathcal{L}$, though it need not be a subbimodule of $\mathcal{L}$ unless $M$ is abelian. The \textit{central part} of a vector $\eta \in \L$ is its projection onto $\mathcal{L}^M$. Lemma~\ref{large} below shows that the central part of a vector lying in a subbimodule is the same whether computed in the subbimodule or in the ambient bimodule.

\begin{example} \label{irr}
    If $P$ is an irreducible subfactor of a $\mathrm{II}_1$ factor $M$, then $L^2(M)^P = \mathbb{C}$.
\end{example}

\begin{lemma} \label{large}
    Let $M$ be a tracial von Neumann algebra and $\K$ be a subbimodule of an $M$-bimodule $\L$. Then for all $\eta \in \K$ we have $P_{\K^M}\eta = P_{\L^M} \eta$.
\end{lemma}

\begin{proof}
    Fix $\zeta \in \L^M$. Since $P_\K$ is $M$-bilinear, we have $P_\K\zeta \in \K^M$. Therefore,
    \begin{equation*}
        \inner{\eta - P_{\K^M}\eta}{\zeta} = \inner{\eta - P_{\K^M}\eta}{P_\K \zeta} = 0,
    \end{equation*}
    hence $\eta - P_{\K^M}\eta \perp \L^M$. From $P_{\K^M}\eta  \in \K^M \subseteq \L^M$, we conclude $P_{\K^M}\eta = P_{\L^M} \eta$. 
\end{proof}

We are ready to define the notion of a \textit{left weakly mixing} bimodule.

\begin{definition} \label{weakly mixing}
    Let $M$ and $N$ be tracial von Neumann algebras. An $M$-$N$ bimodule $\H$ is \textit{left weakly mixing} if it satisfies the following equivalent conditions:
    \begin{itemize}
        \item[(1)] $\{0\}$ is the only $M$-$N$ subbimodule of $\H$ with finite $N$-dimension.

        \item[(2)] As an $M$-bimodule, $\H \otimes_N \overline{\H}$ contains no non-zero central vector.

        \item[(3)] For any finite subset $S$ of $\H$,
        \begin{equation*}
            \inf_{u \in \mathcal{U}(M)} \max_{\xi, \eta \in S} \sup_{y \in N_1} |\inner{u \xi y}{\eta}| = 0.
        \end{equation*}

        \item[(4)] There exists a net $(u_i)_{i \in I} \subseteq \mathcal{U}(M)$ such that
        \begin{equation*}
            \lim_{i \in I} \sup_{y \in N_1} |\inner{u_i \xi y}{\eta}| = 0
        \end{equation*}
        for all $\xi, \eta \in \H$.
    \end{itemize}
    For further equivalent definitions and details, we refer to \cite[Appendix A.2]{Bo14}.
\end{definition}

\section{Detecting the weakly mixing property via a single vector} \label{sec-encoding}

Definition~\ref{weakly mixing} characterizes the weak mixing property of a bimodule $\H$ in terms of the entire space $\H$. In practice, however, $\H$ is often generated by a single cyclic vector, as is the case for $\mathcal{H}(\Phi)$. The following lemma shows that, in this setting, weak mixing is entirely encoded by a condition on this single generator.

\begin{lemma} \label{encoding}
    Let $M$ and $N$ be tracial von Neumann algebras and $\xi$ be a left $N$-bounded cyclic vector of an $M$-$N$ bimodule $\H$. Then the following are equivalent:
    \begin{itemize}
        \item[(a)] $\H$ is left weakly mixing as an $M$-$N$ bimodule.

        \item[(b)] $\xi \otimes \overline{\xi}$ has no central part in the $M$-bimodule $\H \otimes_N \overline{\H}$.
    \end{itemize}
\end{lemma}

\begin{proof}
    The implication (a)$\Rightarrow$(b) is trivial. For the converse, suppose (a) fails to hold so that there exists a nonzero $M$-$N$ subbimodule $\K$ of $\H$ with finite $N$-dimension. By \cite[Lemma A.1]{Va07}, for any $\e > 0$, there exists a projection $z \in \mathcal{Z}(N)$ such that $\tau(z) > 1-\e$ and $\K z$ is finitely generated as a right $N$-module. By replacing $\K$ with $\K z$, we may assume that $\K$ is finitely generated as a right $N$-module.
    
    Choose $n \in \mathbb{N}$ and a right $N$-linear isomorphism $U : \mathcal{K} \rightarrow p L^2(N)^{\oplus n}$ where $p$ is a projection in $M_n(N)$. The left action of $M$ on $\K$ induces a unital $*$-homomorphism $\varphi: M \rightarrow p M_n(N) p$ such that
    \begin{equation*}
        U(x \eta) = \varphi(x) U\eta
    \end{equation*}
    for all $x \in M$ and $\eta \in \K$. Let $\zeta_j = U^{-1}(p e_j) \in \K$ where $e_j$ denotes the $j$-th coordinate vector
    \begin{equation*}
        e_j = (0,\cdots, \widehat{1}, \cdots, 0) \in L^2(N)^{\oplus n}.
    \end{equation*}
    We make the following claim:
    \begin{center}
        $x \zeta_i = \sum_{j=1}^n \zeta_j \varphi(x)_{ji}$ for all $x \in M$ and $i=1,2,\ldots,n$.
    \end{center}
    Indeed, the assertion follows from applying the right $N$-linear map $U^{-1}$ to
    \begin{equation*}
        \varphi(x)(pe_i) = p \varphi(x) e_i = p \left( \sum_{j=1}^n e_j \cdot \varphi(x)_{ji}\right) = \sum_{j=1}^n pe_j \cdot \varphi(x)_{ji}.
    \end{equation*}
    
    Now consider a vector
    \begin{equation*}
        \zeta = \sum_{i=1}^n \zeta_i \otimes \overline{\zeta_i} \in \K \otimes_N \overline{\K} \subseteq \H \otimes_N \overline{\H}.
    \end{equation*}
    It is clear that $\zeta$ is left bounded. Furthermore, $\zeta$ is central. Indeed, by the claim,
    \begin{equation*}
        x \zeta = \sum_{i,j=1}^n \zeta_j \varphi(x)_{ji} \otimes \overline{\zeta_i} = \sum_{i,j=1}^n \zeta_j \otimes \overline{\zeta_i \varphi(x)_{ji}^*} = \sum_{i,j=1}^n \zeta_j \otimes \overline{\zeta_i \varphi(x^*)_{ij}}
    \end{equation*}
    for all $x \in M$. Again by the claim, we obtain
    \begin{equation*}
        x \zeta = \sum_{j=1}^n \zeta_j \otimes \overline{x^* \zeta_j} = \zeta x.
    \end{equation*}
    
    It remains to show that $\xi \otimes \overline{\xi}$ has a nonzero central part, i.e.,
    \begin{equation*}
        \xi \otimes \overline{\xi} \notin \left( \H \otimes_N \overline{\H} \right) \ominus \left( \H \otimes_N \overline{\H} \right)^M.
    \end{equation*}
    Since $\zeta \in \left( \H \otimes_N \overline{\H} \right)^M$, the desired assertion follows if we can prove $\inner{\zeta}{\xi \otimes \overline{\xi}} \neq 0$. Let $\eta = U P_\K \xi \in p L^2(N)^{\oplus n}$. Since $\xi$ is left $N$-bounded, so is $\eta$. Moreover, $\eta \neq 0$ because $P_\K \xi = U^{-1} \eta$ is a cyclic vector of $\K$.
    
    On the other hand, we have
    \begin{equation*}
        \inner{\zeta}{\xi \otimes \overline{\xi}} = \inner{\zeta}{P_\K \xi \otimes \overline{P_\K \xi}} = \sum_{i=1}^n \inner{pe_i \otimes \overline{pe_i}}{\eta \otimes \overline{\eta}}.
    \end{equation*}
    Set $a_i = \inner{pe_i}{\eta}_N \in N$. At least one of $a_1,\ldots,a_n$ is nonzero since $\eta \in p L^2(N)^{\oplus n}$ is nonzero. From
    \begin{equation*}
        \inner{\zeta}{\xi \otimes \overline{\xi}} = \sum_{i=1}^n \inner{pe_i \otimes \overline{pe_i}}{\eta \otimes \overline{\eta}} = \sum_{i=1}^n \inner{\overline{pe_i}}{a_i \overline{\eta}} = \sum_{i=1}^n \inner{\eta a_i^*}{pe_i},
    \end{equation*}
    we conclude
    \begin{equation*}
        \inner{\zeta}{\xi \otimes \overline{\xi}} = \sum_{i=1}^n \tau(\inner{\eta a_i^*}{pe_i}_N) = \sum_{i=1}^n \tau(a_i \inner{\eta}{pe_i}_N) = \sum_{i=1}^n \tau(a_i a_i^*) \neq 0. \qedhere
    \end{equation*}
\end{proof}

\begin{proposition} \label{central}
    Let $\Phi : M \rightarrow N$ be a subtracial completely positive map and $P$ be a von Neumann subalgebra of $M$. Then the following are equivalent:
    \begin{itemize}
        \item[(a)] $\overline{P \xi_{\Phi} N}$ is left weakly mixing as a $P$-$N$ bimodule.
        \item[(b)] $\xi_{\Phi^\circ \Phi}  \in \H(\Phi ^ \circ \Phi) \ominus \H(\Phi ^* \circ \Phi)^P$.
    \item[(c)] We have
    \begin{equation*}
        \inf_{u \in \mathcal{U}(P)} \norm{\Phi(x u x')}_1 =0
    \end{equation*}
    for all $x, x' \in P$.
\end{itemize}
\end{proposition}
\begin{proof}
    We first prove (a)$\Leftrightarrow$(b). By Lemma \ref{encoding}, $\H_0 = \overline{P\xi_\Phi N}$ is left weakly mixing if and only if $\xi_\Phi \otimes \overline{\xi_\Phi}$ has no central part in the $P$-bimodule $\H_0 \otimes_N \overline{\H_0}$. By Remark \ref{embed} and Lemma \ref{large}, this is equivalent to (b). Next we show (a)$\Leftrightarrow$(c). By definition, (a) is equivalent to
    \begin{equation*}
        \inf_{u \in \mathcal{U}(P)} \sup_{z \in N_1} |\inner{u(x \otimes y)z}{x' \otimes y'}| = 0, \quad \forall x, x' \in P, \ y,y' \in N.
    \end{equation*}
    This is again equivalent to
    \begin{equation*}
        \inf_{u \in \mathcal{U}(P)} \sup_{z \in N_1} \left| \tau(z^* y^* \Phi(x^* u x') y')\right| = 0 , \quad \forall x, x' \in P, \ y,y' \in N.
    \end{equation*}
    By the duality of $L^1(N)$ and $N = L^\infty(N)$, (a) is equivalent to
    \begin{equation*}
        \inf_{u \in \mathcal{U}(P)} \norm{y^* \Phi(x^* u x') y'}_1 =0, \quad \forall x, x' \in P, \ y,y' \in N.
    \end{equation*}
    Clearly, this is the same as condition (c).
\end{proof}

\section{Proof of the main result} \label{sec-proof}

In this section, we prove Theorem \ref{main} by establishing the cycle (a)$\Rightarrow$(b)$\Rightarrow$(c)$\Rightarrow$(a).

\subsection{Proof of (a)$\Rightarrow$(b): Pimsner-Popa type inequality}

The following proposition is indeed more general than the implication (a)$\Rightarrow$(b). We remark that the inequality \eqref{ppt-ineq} below plays a role analogous to the Pimsner-Popa inequality for conditional expectations, which was discussed in \cite{PP86}.

\begin{proposition} \label{ppt}
    Let $\Phi : M \rightarrow N$ be a subtracial completely positive map between tracial von Neumann algebras and $P$ be a diffuse von Neumann subalgebra of $M$. Suppose there exists $a \in P$ such that $E_{P' \cap M} (a) \neq 0$ and
    \begin{equation} \label{ppt-ineq}
        (E_P \circ \Phi^* \circ \Phi)(x) \geq a^*xa
    \end{equation}
    for all $x \in P_+$. Then there exists a finite subset $F$ of $P$ with $\tau(x) = 0$ for all $x \in F$ such that 
    \begin{equation*}
        \inf_{u \in \mathcal U(P)} \max_{x \in F} \|\Phi(uxu^*)\|_2 >0.
    \end{equation*}
\end{proposition}

\begin{remark}
    If $P$ is an irreducible subfactor of $M$ as in Theorem \ref{main}, $E_{P' \cap M} (a) \neq 0$ if and only if $\tau(a) \neq 0$. This proves the implication (a)$\Rightarrow$(b) of Theorem \ref{main}.
\end{remark}

To prove Proposition \ref{ppt}, fix $n \in \mathbb{N}$ with $n \norm{E_{P' \cap M} (a)}^2 > 2 {\norm{\Phi(1)}}^2_2$. Because $P$ is diffuse, there exist projections $e_1,\cdots, e_n \in P$ with $\sum_i e_i = 1$ and $\tau(e_i) = 1/n$ for all $i$. We show that a finite set
\begin{equation*}
    F = \{e_i - \tau(e_i) : 1 \leq i \leq n\} \subseteq P \cap \ker\tau
\end{equation*}
satisfies the desired inequality
\begin{equation} \label{starbucks}
    \inf_{u \in \mathcal{U}(P)} \max_{1 \leq i \leq n} \norm{\Phi(u e_i u^*) - \tau(e_i) \Phi(1) }_2 > 0.
\end{equation}
If the left-hand side of \eqref{starbucks} equals zero, there exists $u \in \mathcal{U}(P)$ such that
\begin{equation} \label{restriction}
    \sum_{i=1}^n \norm{\Phi(u e_i u^*)}_2^2 \leq 2 {\norm{\Phi(1)}}_2^2  \cdot \sum_{i=1}^n \tau(e_i)^2 = \frac{2{\norm{\Phi(1)}}_2^2 }{n} < \norm{E_{P' \cap M}(a)}_2^2.
\end{equation}

On the other hand, we have
\begin{equation*}
    \sum_{i=1}^n \norm{\Phi(u e_i u^*)}_2^2 = \sum_{i=1}^n \tau(\Phi(u e_i u^*)^2) = \sum_{i=1}^n \tau(u e_i u^* (\Phi^* \circ \Phi) (u e_i u^*)).
\end{equation*}
Since $u e_i u^* \in P$ for each $i$, it follows that
\begin{equation*}
    \sum_{i=1}^n \norm{\Phi(u e_i u^*)}_2^2 = \sum_{i=1}^n \tau(u e_i u^* (E_P \circ \Phi^* \circ \Phi) (u e_i u^*)).
\end{equation*}
From the assumed inequality \eqref{ppt-ineq}, we obtain
\begin{equation*}
    \sum_{i=1}^n \norm{\Phi(u e_i u^*)}_2^2 \geq \sum_{i=1}^n \tau(u e_i u^* a^* u e_i u^* a) = \sum_{i=1}^n \norm{e_i u^* a u e_i}_2^2 = \norm{\sum_{i=1}^n e_i u^* a u e_i}_2^2.
\end{equation*}
As we have
\begin{equation*}
    \sum_{i=1}^n e_i y e_i = E_{(\oplus_i \mathbb{C} e_i)'\cap M}(y)
\end{equation*}
for all $y \in M$, and $P ' \cap M \subseteq (\oplus_i \mathbb{C} e_i)'\cap M$, we obtain
\begin{equation*}
    \sum_{i=1}^n \norm{\Phi(u e_i u^*)}_2^2 \geq \norm{E_{P' \cap M} (u^* au )}_2^2 = \norm{E_{u (P'\cap M) u^*}(a)}_2^2 = \norm{E_{P' \cap M}(a)}_2^2,
\end{equation*}
the desired contradiction with \eqref{restriction}.

\subsection{Proof of (b)$\Rightarrow$(c)}
We will prove the contrapositive of (b)$\Rightarrow$(c) by proving the following more general proposition.

\begin{proposition} \label{bc}
    Let $\Phi: M \rightarrow N$ be a subtracial completely positive map between tracial von Neumann algebras and $P$ be a von Neumann subalgebra of $M$ that is itself a factor. If $\H_0 = \overline{P \xi_\Phi N}$ is not left weakly mixing as a $P$-$N$ bimodule, then there exists $a \in P$ such that $\tau(a) \neq 0$ and
    \begin{equation*}
        (E_P \circ \Phi^* \circ \Phi)(x) \geq  a^* x a
    \end{equation*}
    for all $x \in P_+$.
\end{proposition}

For the proof of Proposition \ref{bc}, observe from Proposition \ref{central} that the $P$-central part $z$ of $\xi_{\Phi^* \circ \Phi}$ in $\H_0 \otimes_N \overline{\H_0}$ is non-zero. Set $\zeta = z / \norm{z} \in \mathcal{H}(\Phi^* \circ \Phi)^P$ and consider the tracial state
\begin{equation*}
    \tau_0 : P \rightarrow \mathbb{C} : x \mapsto \inner{\zeta}{\zeta x}.
\end{equation*}
Because $P$ is a factor, we deduce $\tau_0 = \tau|_P$. In particular, we have
\begin{equation*}
    \norm{\zeta x} = \sqrt{\inner{\zeta}{\zeta (xx^*)}} = \sqrt{\tau(xx^*)} = \norm{x}_2
\end{equation*}
for all $x \in P$ hence there is a well-defined $P$-bilinear isomorphism
\begin{equation*}
    \alpha : L^2(P) \rightarrow \mathcal{K} : = \overline{\zeta P} \leq \mathcal{H}(\Phi^* \circ \Phi) 
\end{equation*}
such that $\alpha (\widehat{x}) = \zeta x$ for all $x \in P$.
    
Note that $\eta= P_\K (\xi_{\Phi^* \circ \Phi})$ is a left $P$-bounded vector of $\K$. Therefore, $\alpha^{-1}(\eta)$ is a left bounded vector in the standard representation $L^2(P)$ of $P$, hence $ a : = L_{\alpha^{-1}(\eta)} \in P$. Moreover, we have
\begin{equation*}
    \tau(a) = \inner{\widehat{1}}{a \widehat{1}} = \inner{\widehat{1}}{\alpha^{-1}(\eta)} = \inner{\alpha(\widehat{1})}{\eta} = \inner{\zeta}{\eta}.
\end{equation*}
Since $\zeta \in \K$, we deduce that
\begin{equation*}
    \tau(a) = \inner{\zeta}{\xi_{\Phi^* \circ \Phi}} = \norm{z} \neq 0.
\end{equation*}
    
By Lemma \ref{sub}, there exists a completely positive map $\Phi_\K : P \rightarrow P$ such that
\begin{equation*}
    \tau(\Phi_\K(x) y ) = \inner{\eta}{x \eta y}
\end{equation*}
for all $x,y \in P$ and $E_P \circ \Phi ^ * \circ \Phi \geq \Phi_\K$. Because
\begin{equation*}
    \tau(\Phi_\K(x)y)  
    = \inner{\alpha^{-1}(\eta)}{x \alpha^{-1}(\eta)y} = \inner{a \widehat{1}}{(xa) \widehat{1} y} = \tau(a^* x a\, y)
\end{equation*}
for all $x,y \in P$, we deduce $\Phi_\K = a^* (\cdot) a$ and obtain the desired inequality.

\subsection{Proof of (c)$\Rightarrow$(a)\label{new}: Local quantization of bimodules} The proof depends on the following theorem and the integration trick discussed in \cite{QIT26}. This special case, where $\H=L^2(M)$ for some $ P \subseteq M$, amounts to Popa's Local Quantization Principle \cite[Lemma A.1.1]{Po94}.

\begin{theorem}[\cite{QIT26}, Theorem 3.2] \label{QIT}
    Let $P$ be a tracial von Neumann algebra of type $\mathrm{II}_1$ and $\H$ be a $P$-bimodule. Then for every $\varepsilon > 0$ and finite subset $S$ of $\H \ominus \H^P$, there exists a partition of unity $\{p_i\}_{i=1}^m$ in $P$ such that
    \begin{equation*}
        \left\|\sum_{i=1}^m p_i \zeta p_i\right\| < \varepsilon
    \end{equation*}
    for every $\zeta \in S$.
\end{theorem}

Assume the condition (c) of Theorem \ref{main}, i.e., $\H_0 = \overline{P \xi_\Phi N}$ is left weakly mixing as a $P$-$N$ bimodule. Fix a finite subset $F$ of $M_1 \cap \ker \tau$. We must show
\begin{equation*}
    \inf_{u \in \mathcal{U}(P)} \max_{x \in F} \norm{\Phi(u x u^*)}_2 = 0.
\end{equation*}
By Example \ref{irr}, we have
\begin{equation*}
    F \subseteq \ker \tau \subseteq L^2(M) \ominus \mathbb{C} = L^2(M) \ominus L^2(M)^P.
\end{equation*}
By the assumption and Proposition \ref{central}, we also have
\begin{equation*}
    \xi: = \xi_{\Phi^* \circ \Phi} \in \mathcal{H}(\Phi^* \circ \Phi) \ominus \H(\Phi^* \circ \Phi)^P.
\end{equation*}
Therefore, for the $P$-bimodule
\begin{equation*}
    \H = L^2(M) \oplus \H(\Phi^* \circ \Phi),
\end{equation*}
we have
\begin{equation*}
    S : = \{x \oplus \xi  : x \in F\} \subseteq \H \ominus \H^P.
\end{equation*}

Fix $\varepsilon_0>0$. By Theorem \ref{QIT}, there is a partition of unity $\{p_i\}_{i=1}^m \subseteq \mathcal{P}(P)$ with
\begin{equation*}
    \sum_{x \in F} \norm{\sum_{i=1}^m p_i (x \oplus \xi) p_i }^2< \varepsilon_0.
\end{equation*}
Let us denote $L = \max \{1, \norm{\Phi(1)}\} > 0$ so that we have
\begin{equation} \label{app}
    \norm{\Phi(1)} \sum_{x \in F} \sum_{i=1}^m \norm{p_i x p_i}_2^2  + |F| \sum_{i=1}^m \norm{p_i \xi p_i}^2 \leq L \sum_{x \in F} \norm{\sum_{i=1}^m p_i (x \oplus \xi) p_i }^2< L \varepsilon_0.
\end{equation}

Denote the Haar measure on $\mathbb{T}^m$ by $\mu$. For each $\e = (\e_1,\cdots, \e_m) \in \mathbb{T}^m$, set
\begin{equation*}
    u_\e = \sum_{i=1}^m \e_i p_i \in \mathcal{U}(P),
\end{equation*}
and consider an integral
\begin{equation*}
    I = \int_{\mathbb{T}^m} \sum_{x \in F} \norm{\Phi(u_\e x u_\e^*)}_2^2 \ d \mu (\e).
\end{equation*}
If $I < L\e_0$, since $\mu$ is a probability measure, we have
\begin{equation*}
    \left(\inf_{u\in \mathcal{U}(P)} \max_{x \in F} \norm{\Phi(u x u ^*)}_2 \right)^2 \leq \sum_{x \in F} \norm{\Phi(u_\e x u_\e^*)}_2^2 < L \varepsilon_0
\end{equation*}
for some $\e \in \mathbb{T}^m$. Since $\varepsilon_0 > 0$ was arbitrary and $L$ is fixed, it follows that
\begin{equation*}
    \inf_{u\in \mathcal{U}(P)} \max_{x \in F} \norm{\Phi(u x u ^*)}_2 = 0
\end{equation*}
and we are done.

In order to show $I < L\e_0$, observe
\begin{equation*}
    \begin{split}
        I 
        &= \int_{\mathbb{T}^m} \sum_{x \in F} \norm{\Phi(u_\e x u_e^*)}_2^2 \\ 
        &= \sum_{x \in F} \int_{\mathbb{T}^m} \tau(\Phi(u_\e x^* u_e^*)\Phi(u_\e x u_e^*)) \\
        &= \sum_{x \in F} \sum_{i,j,k,\ell}^m \left( \int_{\mathbb{T}^m} \e_i \overline{\e_j} \e_k \overline{\e_\ell} \ d \mu(\e) \right) \cdot \tau(\Phi(p_i x^* p_j)\Phi(p_k x p_\ell)).
    \end{split}
\end{equation*}
The given integral
\begin{equation*}
    \int_{\mathbb{T}^m} \e_i \overline{\e_j} \e_k \overline{\e_\ell} \ d \mu(\e)
\end{equation*}
equals $1$ when
\begin{center}
    `$i=j$ and $k = \ell$' \quad or \quad `$i=\ell$ and $j = k$'
\end{center}
and $0$ otherwise. This shows $I = A + B - C$ where
\begin{equation*}
    \begin{cases}
        A &= \sum_{x \in F} \sum_{i,j=1}^m \tau(\Phi(p_i x^* p_i)\Phi(p_j x p_j)), \\
        B&= \sum_{x \in F} \sum_{i,j=1}^m \tau(\Phi(p_i x^* p_j)\Phi(p_j x p_i)), \\
        C&= \sum_{x \in F} \sum_{i=1}^m \tau(\Phi(p_i x^* p_i)\Phi(p_i x p_i)).
    \end{cases}
\end{equation*}
We have $I \leq A+B$ from
\begin{equation*}
    C = \sum_{x \in F} \sum_{i=1}^m \norm{\Phi(p_i x p_i)}_2^2 \geq 0.
\end{equation*}

Recall the Kadison-Schwarz inequality
\begin{equation*}
    \Phi(x)^* \Phi(x) \leq {\norm{\Phi(1)}} \Phi(x^*x), \quad \forall x \in M,
\end{equation*}
which holds for all $2$-positive maps between $C^*$-algebras (\cite[Corollary 2.8]{Choi74}). Since $\Phi$ is subtracial, $\Phi : M \rightarrow N$ can be extended to a bounded linear map $L^2(M) \rightarrow L^2(N)$ with the operator norm $\leq {\norm{\Phi(1)}}^{\frac{1}{2}}$. Therefore,
\begin{equation*}
    A = \sum_{x \in F} \norm{\sum_{i=1}^m \Phi(p_i x p_i)}_2^2 \leq {\norm{\Phi(1)}} \sum_{x \in F} \norm{\sum_{i=1}^m p_i x p_i}_2^2 .
\end{equation*}
Since $p_i M p_i$'s are orthogonal in $L^2(M)$, we obtain an estimate
\begin{equation} \label{esA}
    A \leq \norm{\Phi(1)} \sum_{x \in F} \sum_{i=1}^m \norm{p_i x p_i}_2^2.
\end{equation}
On the other hand, we have
\begin{equation*}
    \tau(\Phi(y^*) \Phi(y)) = \tau(y^* (\Phi^* \circ \Phi) (y)) = \inner{\xi}{y \xi y^*}
\end{equation*}
for all $y \in M$, where $\xi = \xi_{\Phi^* \circ \Phi}$ is the canonical cyclic vector of the $M$-bimodule $\H(\Phi^* \circ \Phi)$. It follows that
\begin{equation*}
    \begin{split}
        B 
        = \sum_{x \in F} \sum_{i,j=1}^m \inner{\xi}{(p_j x p_i) \xi (p_j x p_i)^*}
        = \sum_{x \in F} \inner{ \sum_{i=1}^m p_i \xi p_i }{x \left( \sum_{i=1}^m p_i \xi p_i\right) x^*}.
    \end{split}
\end{equation*}
Because $F \subseteq M_1$, we have an estimate
\begin{equation} \label{esB}
    B \leq \sum_{x \in F} \norm{\sum_{i=1}^m p_i \xi p_i}^2 = |F| \sum_{i=1}^m  \norm{p_i \xi p_i}^2.
\end{equation}
Combining $I \leq A+B$ with \eqref{app}, \eqref{esA} and \eqref{esB}, we conclude
\begin{equation*}
    I \leq \norm{\Phi(1)} \sum_{x \in F} \sum_{i=1}^m \norm{p_i x p_i}_2^2  + |F| \sum_{i=1}^m \norm{p_i \xi p_i}^2 < L\e_0.
\end{equation*}

\section{Application: asymptotic orthogonalization} \label{sec-asym}
From now on, $\omega \in \beta\mathbb{N}\setminus\mathbb{N}$ denotes a fixed
free ultrafilter on $\mathbb{N}$. For a tracial von Neumann algebra
$(M,\tau)$, we write $M^\omega$ for the tracial ultrapower, with its
canonical trace again denoted $\tau$. For the construction and basic properties of the tracial ultrapower, we refer to \cite[Section 5.4]{AP17}. In the following result, $\xi_{\Phi^\omega}$ denotes the canonical cyclic vector of the $M^\omega$-$N^\omega$ bimodule $\mathcal{H}(\Phi^\omega)$ associated with the completely positive map $\Phi^\omega : M^\omega \rightarrow N^\omega$.

\begin{corollary}[asymptotic orthogonalization] \label{asym-otho}
    Let $M$ be a $\mathrm{II}_1$ factor, $P \subseteq M$ an irreducible $\mathrm{II}_1$ subfactor, and $N$ a tracial von Neumann algebra. If $\Phi: M \rightarrow N$ is subtracial and completely positive, then the following are equivalent:
    \begin{itemize}
        \item[(a)] $\overline{P \xi_\Phi N}$ is left weakly mixing as a $P$-$N$ bimodule.

        \item[(b)] For any separable von Neumann subalgebra $B$ of $M$, there exists $u \in \mathcal{U}(P^\omega)$ such that
        \begin{equation*}
            \xi_{\Phi^\omega} N^\omega \perp u (B \ominus \mathbb{C}) u^* \xi_{\Phi^\omega} \subseteq \mathcal{H}(\Phi^\omega).
        \end{equation*}
    \end{itemize}
\end{corollary}
\begin{proof}
    Assume (a) first. By Theorem \ref{main}, for any separable von Neumann subalgebra $B$ of $M$, we may choose a sequence $(u_n)_{n \in \mathbb{N}} \subseteq \mathcal{U}(P)$ such that
    \begin{equation*}
        \lim_{n \rightarrow \omega} \norm{\Phi(u_n x u_n^*)}_2 = 0
    \end{equation*}
    for all $x \in B$ with $\tau(x) =0$. It follows that for any bounded $(y_n)_{n \in \mathbb{N}} \subseteq N$, we have
    \begin{equation*}
        \lim_{n\rightarrow \omega} \inner{\xi_\Phi y_n} {(u_n x u_n^*)\xi_\Phi} = \lim_{n \rightarrow \omega} \tau(y_n^* \Phi(u_n x u_n^*)) = 0.
    \end{equation*}
    Therefore, $u = (u_n)_n^\omega \in \mathcal{U}(P^\omega)$ satisfies the desired orthogonality in $\H(\Phi^\omega)$.

    Now assume (b). Fix a finite subset $F$ of $M$ such that $\tau(x) = 0$ for all $x \in F$. Let $B$ be the von Neumann subalgebra of $M$ generated by $F$. Applying (b) to this separable von Neumann subalgebra of $M$, we obtain a unitary $u \in \mathcal{U}(P^\omega)$ satisfying
    \begin{equation*}
        \xi_{\Phi^\omega} N^\omega \perp u F u^* \xi_{\Phi^\omega} \subseteq \mathcal{H}(\Phi^\omega).
    \end{equation*}
    Write $u = (u_n)_n^\omega$ for $u_n \in \mathcal{U}(P)$. It follows that
    \begin{equation*}
        \lim_{n\rightarrow \omega} \norm{\Phi(u_n x u_n^*)}_2^2 = 
        \inner{\xi_{\Phi^\omega} (\Phi(u_n x u_n^*))_n^\omega}{(u_n x u_n^*)_n^\omega \xi_{\Phi^\omega}} = 0
    \end{equation*}
    for all $x \in F$. This proves
    \begin{equation*}
        \inf_{u \in \mathcal{U}(P)} \max_{x \in F} \norm{\Phi(u x u^*)}_2 = 0
    \end{equation*}
    and (a) follows from Theorem \ref{main}.
\end{proof}

The implication (a)$\Rightarrow$(b) of Corollary \ref{asym-otho}, as discussed in the introduction, was already known to Popa \cite{Po19} in the special case where $\Phi$ is given by a conditional expectation. To be precise, let $M$ be a $\mathrm{II}_1$ factor, $N \subseteq M$ a von Neumann subalgebra such that $M \nprec_M N$ (equivalently, $\mathcal{H}(E_N) = L^2(M)$ is left weakly mixing as an $M$-$N$ bimodule), and $B \subseteq M$ a separable von Neumann subalgebra. By \cite[Corollary 1.2]{Po19}, there exists a unitary $u \in \mathcal{U}(M^\omega)$ such that
\begin{equation*}
    \tau(uxu^*y) = \tau(x)\tau(y)
\end{equation*}
for all $x \in B$ and $y \in N^\omega$, or equivalently,
\begin{equation*}
    u(B \ominus \mathbb{C})u^* \perp N^\omega \subseteq L^2(M^\omega).
\end{equation*}
Under the identification $\mathcal{H}(E_N^\omega) = L^2(M^\omega)$ as an $M^\omega$-$N^\omega$ bimodule, this recovers the special case of Corollary \ref{asym-otho} where $P = M$ and $\Phi = E_N$.

\section{More Applications} \label{sec-moreapp}

We collect several further consequences of Theorem \ref{main}. Corollary \ref{appcor1} shows that weak mixing passes from an irreducible subfactor $P \subseteq M$ up to $M$ itself. We then use Theorem \ref{main} to give an alternative proof of \cite[Lemma 2.3]{Po19} (Corollary \ref{cond}), bypassing the original patching argument. Finally, Corollary \ref{suff} isolates a sufficient condition for the failure of the flattening phenomenon when $\Phi$ is a self-map $M \rightarrow M$, and we use it to recover the classical fact that the Pimsner-Popa inequality obstructs asymptotic orthogonalization for finite-index subfactors.

\begin{corollary} \label{appcor1}
    Let $\Phi: M \rightarrow N$ be a subtracial completely positive map and $P \subseteq M$ be an irreducible inclusion of $\mathrm{II}_1$ factors. If $\overline{P \xi_\Phi N}$ is left weakly mixing as a $P$-$N$ bimodule, then $\H(\Phi)$ is left weakly mixing as an $M$-$N$ bimodule. Equivalently, whenever
    \begin{equation*}
        \inf_{u \in \mathcal{U}(P)} \norm{\Phi(x u x')}_1 =0
    \end{equation*}
    for all $x,x' \in P$, we have
    \begin{equation*}
        \inf_{u \in \mathcal{U}(M)} \norm{\Phi(x u x')}_1 =0
    \end{equation*}
    for all $x, x' \in M$.
\end{corollary}
\begin{proof}
    By assumption and Theorem \ref{main}, we have
    \begin{equation*}
        \inf_{u \in \mathcal{U}(P)} \max_{x \in F}\norm{\Phi(uxu^*)}_2 = 0
    \end{equation*}
    for any finite subset $F \subseteq M$ such that $\tau(x) = 0$ for all $x \in F$. It follows that
    \begin{equation*}
        \inf_{u \in \mathcal{U}(M)} \max_{x \in F}\norm{\Phi(uxu^*)}_2 = 0.
    \end{equation*}
    Again by Theorem \ref{main} (or Corollary \ref{cormain}), the desired assertion is immediate.
\end{proof}
 
Before giving an alternative proof of \cite[Lemma 2.3]{Po19}, we first recall Popa's intertwining-by-bimodules technique as developed in \cite{Po06a,Po06b}. For further details, we also refer to Chapter 17 of the textbook \cite{AP17}.

\begin{definition} \label{def-xcorner}
    Let $(M,\tau)$ be a tracial von Neumann algebra, and $P, N \subseteq M$ be its von Neumann subalgebras. We write $P \prec_M N$ and say that a corner of $P$ can be intertwined into $N$ inside $M$ if the following
    equivalent conditions hold:
    \begin{enumerate}
        \item There exists no net $(u_i) \subseteq \mathcal{U}(P)$ such that
        \begin{equation*}
            \lim_i \bigl\| E_N(x^* u_i y) \bigr\|_2 = 0
        \end{equation*}
        for all $x,y \in M$.
        
        \item There exists a non-zero $P$-$N$ subbimodule of
        $L^2(M)$ with finite right $N$-dimension, i.e., $L^2(M)$ is not left weakly mixing as a $P$-$N$ bimodule.
      
        \item There exist $n \in \mathbb{N}$, a projection
        $q \in M_n(\mathbb{C}) \otimes N$, a non-zero partial isometry
        $v \in M_{1,n}(\mathbb{C}) \otimes M$ and a normal unital
        $*$-homomorphism $\theta \colon P \to q\bigl(M_n(\mathbb{C}) \otimes N\bigr)q$
        such that $v^*v \leq q$ and $xv = v\theta(x)$ for every $x \in P$.
    \end{enumerate}
\end{definition}

\begin{corollary}[\cite{Po19}, Lemma 2.3] \label{cond}
    Let $M$ be a $\mathrm{II}_1$ factor, $N \subseteq M$ a von Neumann subalgebra, and $P \subseteq M$ an irreducible subfactor such that $P \not\prec_M N$. Given any finite set $F \subseteq M$ such that $\tau(x) = 0$ for all $x \in F$, we have
    \begin{equation*}
        \inf_{u \in \mathcal{U}(P)} \max_{x \in F} \norm{E_N(u x u^*)}_2 = 0
    \end{equation*}
\end{corollary}
\begin{proof}
    By the assumption $P \not\prec_M N$, $L^2(M)$ is left weakly mixing as a $P$-$N$ bimodule. With the canonical $M$-$N$ bilinear isomorphism
    \begin{equation*}
        \mathcal{H}(E_N) \rightarrow L^2(M) : x \otimes y \mapsto xy
    \end{equation*}
    (see \cite[Section 13.1.2]{AP17} for details), we can identify $\overline{P \xi_{E_N} N}$ as a $P$-$N$ subbimodule of $L^2(M)$. It follows that $\overline{P \xi_{E_N} N}$ is left weakly mixing and the desired assertion is immediate from Theorem \ref{main}.
\end{proof}

Let $M$ be a $\mathrm{II}_1$ factor and $N \subseteq M$ be a subfactor with $[M:N] < \infty$. By \cite[Proposition 2.1]{PP86}, we have the Pimsner-Popa inequality
\begin{equation*}
    E_N(x) \geq [M:N]^{-1} x
\end{equation*}
for all $x \in M_+$. Assume for the contrary that
\begin{equation*}
    \inf_{u \in \mathcal{U}(M)} \max_{x \in F} \norm{E_N(u x u^*) -\tau(x)}_2 =0
\end{equation*}
for any finite subset $F$ of $M$. In particular, for any projection $p \in \mathcal{P}(M)$, we obtain
\begin{equation*}
    \inf_{u \in \mathcal{U}(M)} \norm{E_N(u p u^*) - \tau(p)}_2 = 0.
\end{equation*}
Fix $u \in \mathcal{U}(M)$ with $\norm{E_N(u p u^*)}_2 \leq 2 \tau(p)$. By the Pimsner-Popa inequality,
\begin{equation*}
    2 \tau(p) \geq \norm{E_N(u p u^*)}_2 \geq [M:N]^{-1} \norm{upu^*}_2 = [M:N]^{-1}\sqrt{\tau(p)}.
\end{equation*}
It follows that $0 < [M:N]^{-1} \leq 2 \sqrt{\tau(p)}$ for all $p \in \mathcal{P}(M)$, a contradiction with $\tau(\mathcal{P}(M)) = [0,1]$. The converse is also true when $N$ is an irreducible subfactor of $M$.

\begin{corollary} \label{laundry}
    Suppose that $N$ is an irreducible subfactor of a $\mathrm{II}_1$ factor $M$. Then $[M : N] = \infty$ if and only if there exists a net $(u_i)_{i \in I} \subseteq \mathcal{U}(M)$ satisfying
    \begin{equation*}
        \lim_{i \in I} \norm{E_N(u_i x u_i^*) - \tau(x)}_2 = 0
    \end{equation*}
    for all $x \in M$.
\end{corollary}

\begin{proof}
    We have already observed the `only if' part. For the converse, assume $[M:N]<\infty$. In particular, $L^2(M)$ is an $M$-$N$ subbimodule of itself whose right $N$-dimension equals $[M:N] < \infty$. It follows that $L^2(M) = \overline{M \xi_{E_N} N}$ is not left weakly mixing. Now Theorem \ref{main} completes the proof. 
\end{proof}

The argument preceding Corollary \ref{laundry} shows concretely how the Pimsner-Popa inequality prevents $E_N$ from \textit{flattening}. 

Let $M$ be a separable $\mathrm{II}_1$ factor. The space $\mathcal{S}(M)$ of von Neumann subalgebras of $M$, equipped with the \textit{Maréchal topology}, is a Polish space. Moreover, for $N_n,N \in \mathcal{S}(M)$, we have
\begin{equation*}
    N_n \rightarrow N
\end{equation*}
in the Maréchal topology if and only if
\begin{equation*}
    e_{N_n} \rightarrow e_N
\end{equation*}
in the strong operator topology of $B(L^2(M))$. We refer to \cite{HW98} for further discussion of the Maréchal topology.

Suppose that $N_0 \subseteq M$ is an irreducible subfactor with $[M:N_0] = \infty$. By Corollary~\ref{laundry}, there exists a sequence $(u_n)_{n \in \mathbb{N}} \subseteq \mathcal{U}(M)$ such that
\begin{equation} \label{gold}
    \lim_{n \rightarrow \infty}
    \norm{E_{N_0}(u_n^*xu_n)-\tau(x)}_2=0,
    \qquad x\in M.
\end{equation}
Set $N_n=u_nN_0u_n^*$. Since
\begin{equation*}
    E_{N_n}(x)
    =
    u_nE_{N_0}(u_n^*xu_n)u_n^*,
\end{equation*}
condition \eqref{gold} implies
\begin{equation*}
    \norm{E_{N_n}(x)-\tau(x)}_2 \rightarrow 0,
    \qquad x\in M.
\end{equation*}
Equivalently, $e_{N_n}\rightarrow e_{\mathbb{C}}$ in the strong operator topology of $B(L^2(M))$. Thus
\begin{equation*}
    N_n\rightarrow \mathbb{C}
\end{equation*}
in the Maréchal topology. Therefore, if we denote
\begin{equation*}
    \mathcal{O}(N_0) = \{uN_0u^*:u\in\mathcal{U}(M)\}
\end{equation*}
the unitary orbit of $N_0$ under the conjugation action of $\mathcal{U}(M)$ on $\mathcal{S}(M)$, we have
\begin{equation*}
    \mathbb{C}
    \in
    \overline{\mathcal{O}(N_0)}
    \setminus
    \mathcal{O}(N_0).
\end{equation*}
In particular, $\mathcal{O}(N_0)$ is not closed in the Maréchal topology.

We next show that the converse fails, i.e., $\mathcal{O}(N_0)$ might fail to be closed even if $[M:N_0]<\infty$. Let $R=\otimes_{n=1}^{\infty} M_2(\mathbb C)$ be the hyperfinite
$\mathrm{II}_1$ factor and define
\begin{equation*}
    \gamma
    =
    \bigotimes_{n=1}^{\infty}\operatorname{Ad}(v)
    \in \operatorname{Aut}(R),
    \qquad
    v
    =
    \begin{pmatrix}
        1&0\\
        0&-1
    \end{pmatrix}.
\end{equation*}
For a matrix
\begin{equation*}
    w
    =
    \begin{pmatrix}
        0&1\\
        1&0
    \end{pmatrix},
\end{equation*}
let $w_n\in R$ denote the element which is $w$ in the $n$-th tensor
component and the identity elsewhere. It is easy to see that
$(w_n)_{n=1}^{\infty}$ is a central sequence in $R$.

Define $M=R\overline\otimes R\simeq R$. Since $vwv^*=-w$, we have
\begin{equation*}
    \norm{(\gamma\otimes\gamma)(w_n\otimes1)-w_n\otimes1}_2
    =
    \norm{-w_n\otimes1-w_n\otimes1}_2
    =
    2.
\end{equation*}
On the other hand, for every $u\in\mathcal U(M)$, we have
\begin{equation*}
    \lim_{n\rightarrow\infty}
    \norm{\operatorname{Ad}(u)(w_n\otimes1)-w_n\otimes1}_2
    =
    0,
\end{equation*}
because $(w_n)_{n=1}^{\infty}$ is a central sequence in $R$. This proves $\gamma\otimes\gamma \in \operatorname{Aut}(M)\setminus\operatorname{Inn}(M)$. Similarly, one can see that $\gamma \otimes \operatorname{id},\ \operatorname{id} \otimes \gamma \in \operatorname{Aut}(M)\setminus\operatorname{Inn}(M)$.

Now consider the two von Neumann subalgebras
\begin{equation*}
    N_0=M^{\gamma\otimes\operatorname{id}},
    \qquad
    N_1=M^{\operatorname{id}\otimes\gamma}
\end{equation*}
of $M$. Observe that $N_0, N_1 \subseteq M$ are irreducible subfactors (cf. \cite[Lemma 6]{NT60}) of
$M$ such that $[M:N_0] = [M:N_1] = 2$ and that the flip automorphism $\theta \in \operatorname{Aut}(M)$ satisfies $\theta(N_0) = N_1$.

Assume, for a contradiction, that $uN_0u^*=N_1$ for some $u\in\mathcal U(M)$. Define
\begin{equation*}
    \beta
    =
    \operatorname{Ad}(u)
    \circ
    (\gamma\otimes\operatorname{id})
    \circ
    \operatorname{Ad}(u^*)
    =
    \operatorname{Ad}(u(\gamma\otimes\operatorname{id})(u^*))
    \circ
    (\gamma\otimes\operatorname{id})
    \in\operatorname{Aut}(M).
\end{equation*}
Observe that
\begin{equation*}
    M^\beta
    =
    uM^{\gamma\otimes\operatorname{id}}u^*
    =
    uN_0u^*
    =
    N_1
    =
    M^{\operatorname{id}\otimes\gamma}.
\end{equation*}
Since $\beta$ and $\operatorname{id}\otimes\gamma$ are involutions, we obtain
\begin{equation*}
    \frac{\operatorname{id}+\beta}{2}
    =
    E_{M^\beta}
    =
    E_{M^{\operatorname{id}\otimes\gamma}}
    =
    \frac{\operatorname{id}+\operatorname{id}\otimes\gamma}{2}.
\end{equation*}
Hence $\beta = \operatorname{id} \otimes \gamma$, i.e.,
\begin{equation*}
    \operatorname{Ad}(u(\gamma\otimes\operatorname{id})(u^*))
    \circ
    (\gamma\otimes\operatorname{id})
    =
    \operatorname{id}\otimes\gamma.
\end{equation*}
Consequently,
\begin{equation*}
    \gamma\otimes\gamma
    =
    (\operatorname{id}\otimes\gamma)
    \circ
    (\gamma\otimes\operatorname{id})^{-1}
    =
    \operatorname{Ad}(u(\gamma\otimes\operatorname{id})(u^*))
    \circ
    (\gamma\otimes\operatorname{id})
    \in
    \operatorname{Inn}(M),
\end{equation*}
which contradicts that $\gamma\otimes\gamma$ is outer. Therefore, $N_0$ and $N_1$ are not inner conjugate.

By \cite[Corollary~3.2]{Connes76}, every automorphism of $M \simeq R$ is approximately inner. Hence there exists a sequence $(u_n)_{n\in\mathbb{N}}\subseteq\mathcal U(M)$ such that
\begin{equation} \label{melatonin}
    \lim_{n\rightarrow\infty}
    \norm{\theta(x)-u_nxu_n^*}_2=0,
    \qquad x\in M.
\end{equation}

Set $N_n=u_nN_0u_n^*$. It remains to show that $N_n \rightarrow N$ in the Maréchal topology. For $\eta\in\operatorname{Aut}(M)$, let $U_\eta$ denote the unitary extension of $\eta$ to $L^2(M)$. Now \eqref{melatonin} shows
\begin{equation*}
    U_{\operatorname{Ad}(u_n)}
    \rightarrow
    U_\theta
\end{equation*}
in the strong operator topology of $B(L^2(M))$. Since these operators are unitary, we also have
\begin{equation*}
    U_{\operatorname{Ad}(u_n^*)}
    =
    U_{\operatorname{Ad}(u_n)}^*
    \rightarrow
    U_\theta^*
    =
    U_{\theta^{-1}}
\end{equation*}
in the strong operator topology. It follows that
\begin{equation*}
    e_{N_n}
    =
    U_{\operatorname{Ad}(u_n)}
    e_{N_0}
    U_{\operatorname{Ad}(u_n^*)}
    \rightarrow
    U_\theta e_{N_0}U_{\theta^{-1}}
    =
    e_{\theta(N_0)}
    =
    e_N
\end{equation*}
in the strong operator topology of $B(L^2(M))$ as desired.

The following corollary places this same mechanism in the general framework of Theorem 1.1. 

\begin{corollary} \label{suff}
    Let $M$ be a $\mathrm{II}_1$ factor and $\Phi : M \rightarrow M$ be a subtracial completely positive map. Suppose there exists $a \in M \setminus \{0\}$ such that
    \begin{equation*}
        \Phi(x) \geq a^* x a
    \end{equation*}
    for all $x \in M_+$. Then for some finite subset $F \subseteq \ker \tau_M$, we have
    \begin{equation*}
        \inf_{u \in \mathcal{U}(M)} \max_{x \in F} \norm{\Phi(u x u^*)}_2 > 0.
    \end{equation*}
\end{corollary}
\begin{proof}
    In view of Theorem \ref{main}, it suffices to find $h \in M_+ \setminus \{0\}$ such that
    \begin{equation*}
        (\Phi^*\circ\Phi)(x) \geq h x h
    \end{equation*}
    for all $x \in M_+$. For all $x,y \in M_+$, we have
    \begin{equation*}
        \tau(y (\Phi^* \circ \Phi)(x)) = \tau(\Phi(y) \Phi(x)) \geq \tau(a^* y a a^* x a) = \tau(y (a a^* x a a^*)).
    \end{equation*}
    Therefore, the desired inequality holds with $h = aa^* \in M_+ \setminus \{0\}$.
\end{proof}

\printbibliography

\end{document}